\documentclass[11pt,reqno]{amsart}

\usepackage[T1]{fontenc}
\usepackage[utf8]{inputenc}
\usepackage{lmodern}
\usepackage{amsmath,amssymb,mathtools}
\usepackage{enumitem}
\usepackage[a4paper,textwidth=6.35in,textheight=9.15in,centering]{geometry}
\usepackage{microtype}
\usepackage[hidelinks]{hyperref}
\usepackage{cleveref}
\hypersetup{
  pdftitle={The Solvable Automorphism Groups of Function Fields over Finite Fields}
}
\usepackage{array}
\usepackage{booktabs}
\usepackage{tabularx}

\allowdisplaybreaks
\setlist{topsep=2pt,itemsep=1pt,parsep=0pt,partopsep=0pt}
\newtheorem{theorem}{Theorem}[section]
\newtheorem{proposition}[theorem]{Proposition}
\newtheorem{lemma}[theorem]{Lemma}
\newtheorem{corollary}[theorem]{Corollary}
\theoremstyle{definition}
\newtheorem{definition}[theorem]{Definition}

\newtheorem{conjecture}[theorem]{Conjecture}

\newtheorem*{conjecture*}{Conjecture}
\theoremstyle{remark}
\newtheorem{remark}[theorem]{Remark}
\crefname{conjecture}{conjecture}{conjectures}
\Crefname{conjecture}{Conjecture}{Conjectures}

\newcommand{\F}{\mathbb{F}}
\newcommand{\Z}{\mathbb{Z}}
\newcommand{\Aut}{\operatorname{Aut}}
\newcommand{\Gal}{\operatorname{Gal}}
\newcommand{\Diff}{\operatorname{Diff}}
\newcommand{\Ram}{\operatorname{Ram}}
\newcommand{\PGL}{\operatorname{PGL}}

\newcommand{\charac}{\operatorname{char}}
\newcommand{\PSL}{\operatorname{PSL}}
\newcommand{\PGU}{\operatorname{PGU}}

\title[Solvable automorphism groups]{On Solvable Automorphism Groups of Function Fields over Finite Fields}
\author{Henning Stichtenoth}
\email{henning@sabanciuniv.edu}
\address{ Sabancı University, Orhanlı, Tuzla, 34956 İstanbul, Turkey}
\author{Chaoping Xing }
\email{xingcp@sjtu.edu.cn }
\address{School of Computer Sciences, Shanghai Jiao Tong University, Shanghai, China 200240}
\author{Zugan Xing}
\email{xingzugan@aliyun.com}
\address{School of Mathematical Sciences,
Shanghai Jiao Tong University,
Shanghai 200240, China}
\date{}

\begin{document}

\begin{abstract}

The classical Hurwitz bound gives a linear upper bound in terms of the genus for automorphism groups of algebraic function fields in characteristic zero, whereas algebraic function fields in positive characteristic can have much larger automorphism groups.

Usually, one considers the case where the constant field is an algebraically closed field. Here we focus on function fields over a finite constant field. It turns out that the situation is very different. We firstly investigate solvable subgroups and  $p$-subgroups of $\Aut(F/\F_q),$ where $p=\charac(\F_q)$ and $F/\F_q$ is a function field of genus $g_F\ge2$ with full constant field $\F_q$. We prove that there are explicit constants $C_q^{(\mathrm{sol})}$ and $C_q^{(p)}$, depending only on $q$, such that, for every  solvable subgroup $G\le\Aut(F/\F_q)$ and every  $p$-subgroup $P\le\Aut(F/\F_q)$,
\[
   |G|\le C_q^{(\mathrm{sol})}(g_F-1),
   \qquad
   |P|\le C_q^{(p)}(g_F-1).
\]
In particular, one may take $C_q^{(\mathrm{sol})}=12q^2(q-1)$ for $q\ge4$ and $C_q^{(p)}=5q$ for every $q$.

In the second part of this article, we give asymptotic lower bounds for solvable automorphism groups. This means, for every $q$ there is an explicit constant $L_q>0$ and a sequence of function fields $F_i/\F_q$ with $g_{F_i}\to\infty$, together with solvable subgroups $G_i\le\Aut(F_i/\F_q)$, such that
\[
   |G_i|\ge L_q(g_{F_i}-1)
\]
for all $i$. For $q\ge5$ one may take $L_q=2q(q-1)/(q-2)$. If $q=r^2$  with $r\ge3$, an unramified tower construction over the Hermitian function field gives the stronger choice $L_q=2r^3(r-1)/(r-2)$. 

Moreover, as an application of our tower construction above, we improve the asymptotic lower bound of Ma and Xing \cite{MaXing2019} for subgroups of $\Aut(F/\F_q)$ whose order is relatively prime to $p$. As another application, we extend our method to obtain explicit linear bounds for $p$-solvable groups in terms of the genus.
\end{abstract}

\maketitle

\section{Introduction}

A classical problem in the theory of algebraic function fields is to study the order of automorphism groups in terms of the genus.

Let $F/K$ be an algebraic function field of one variable of genus $g_F\ge2$, where $K$ is algebraically closed. Then the automorphism group $\Aut(F/K)$ is finite. If $\operatorname{char}(K)=0$, the Hurwitz bound gives
\[
   |\Aut(F/K)|\le84(g_F-1).
\]

However, this bound fails in positive characteristic because wild ramification may occur. The failure can be very large: for the Hermitian function field $H/K$ over an algebraically closed field of characteristic $p>0$, one has
\[
   |\Aut(H/K)|\ge16g_H^4.
\]
Automorphism groups of algebraic function fields over algebraically closed fields of positive characteristic have been widely studied  \cite[Chapter 11]{HirschfeldKorchmarosTorres}.

Let \(q=p^e\) be a prime power and let \(F/\F_q\) have full constant field \(\F_q\). Write
\[
   \Aut(F/\F_q)
   =\{\sigma:F\longrightarrow F:\sigma\text{ is an }\F_q\text{-automorphism}\}.
\]
If $\overline F=F\overline{\F}_q$ where $\overline{\F}_q$ denotes the algebraic closure of $\F_q,$ extension of constants gives an embedding
\[
   \Aut(F/\F_q)\hookrightarrow\Aut(\overline F/\overline{\F}_q),
\]
but not every automorphism of $\overline F/\overline{\F}_q$ descends to $\F_q$.

For abelian subgroups, Nakajima \cite{Nakajima1987} showed that abelian subgroups of $\Aut(\overline F/\overline{\F}_q)$ can have orders growing linearly with the genus. In contrast, it was shown in \cite{MaXing2019}  that, over a fixed finite field $\F_q$, the maximum order of an abelian subgroup of $\Aut(F/\F_q)$ is $o(g_F)$ as $g_F\to\infty$.

For solvable subgroups, Korchm\'aros and Montanucci \cite{KorchmarosMontanucciOrdinary} proved that every solvable subgroup $G\le\Aut(F/K)$ of an ordinary function field (i.e., its $p$-rank is equal to its genus) $F/K$ of genus $g_F\ge2$ over an algebraically closed field $K$ of odd characteristic satisfies
\[
   |G|\le 34(g_F+1)^{3/2}.
\]
In characteristic $2$, Montanucci and Speziali \cite[Corollary~3.11]{MontanucciSpezialiChar2} established the bound
\begin{equation}\label{eq:MS-bound}
   |G|<35(g_F+1)^{3/2}
\end{equation}
when $g_F$ is even. The exponent $3/2$ is asymptotically optimal in odd characteristic \cite{KorchmarosMontanucciOrdinary}. However, these constructions do not yield solvable automorphism groups of order $\Omega(g_F^{3/2})$ over a fixed finite field $\F_q$.




Our first main result shows that, for every prime power $q$, there exists an explicit constant $C_q^{(\mathrm{sol})}>0$ such that
\begin{equation}\label{eq:intro-uniform-solvable}
   |G|\le C_q^{(\mathrm{sol})}(g_F-1)
\end{equation}
for every function field $F/\F_q$ of genus $g_F\ge2$ and every solvable subgroup $G\le\Aut(F/\F_q)$. In particular, one may take
\[
   C_q^{(\mathrm{sol})}=12q^2(q-1),\qquad q\ge4.
\]
A sharper constant is given in Theorem~\ref{cor:Sq-upper}.

For  $p$-subgroups, where $q=p^e$, we obtain a stronger linear upper bound. More precisely, for every $p$-subgroup $P\le\Aut(F/\F_q)$,
\begin{equation}\label{eq:intro-uniform-p}
   |P|\le C_q^{(p)}(g_F-1),
\end{equation}
where one may take $C_q^{(p)}=5q$ for every $q$. Sharper constants $C_q^{(p)}$ are given in Section~\ref{sec:p-subgroup-upper-bounds}.

In the second part, we establish asymptotic lower bounds for solvable automorphism groups.
More precisely, for every $q$ there exists an explicit constant $L_q>0$ and a sequence of function fields \(
   F_i/\F_q\text{ with }g_{F_i}\rightarrow\infty,
\) together with solvable subgroups \(
   G_i\le \Aut(F_i/\F_q),
\) such that
\begin{equation}\label{eq:intro-general-lower}
   |G_i|\ge L_q(g_{F_i}-1)
\end{equation}
for all $i$. Our construction gives, for $q\ge5$,
\[
   L_q=\frac{2q(q-1)}{q-2}.
\]
The small constant fields $q=2,3,4$ require separate constructions,  which are treated in Section~\ref{sec:asymptotic-lower-bounds}. For square $q$, we obtain a stronger bound. If $q=r^2$ with $r\ge3$, an unramified tower over the Hermitian function field gives
\begin{equation}\label{eq:intro-hermitian-lower}
   L_q=\frac{2r^3(r-1)}{r-2}.
\end{equation}
We also consider explicit constructions of function field towers. For $p>2$, an explicit recursive Garcia--Stichtenoth tower gives a construction with solvable subgroups  $ G_i\le \Aut(F_i/\F_q),$ such that
\[
   |G_i|\ge(q-1)(g_{F_i}-1).
\]
Our constructions based on class field towers yield further results.  Proposition~\ref{prop:restricted-lifting} extends from solvable to $p'$-automorphism groups, see \Cref{cor:pprime-lifting}. As an application, we improve the asymptotic lower bound $B_q\ge 2/3$ of Ma and Xing \cite{MaXing2019} to
\[
   B_q\ge \frac{2(q-1)}{q-2},
   \qquad q\ge 5.
\]
See \Cref{thm:Bq-lower} for details.

Our method for bounding solvable automorphism groups also extends to  $p$-solvable groups, whose chief factors are either $p$-groups or $p'$-groups. For every $p$-solvable subgroup $G\le\Aut(F/\F_q)$ with $q\ge4$, we obtain
\[
   |G|\le21q^2(q-1)(g_F-1).
\]
We discuss this result in more detail in \Cref{thm:p-solvable}.

Since Schmid's work in 1938 \cite{Schmid1938}, it has been known that the classical Hurwitz bound fails over algebraically closed fields of positive characteristic. However, it remains unknown whether a linear bound in terms of the genus holds for automorphism groups over a fixed finite field $\F_q$.

The results above establish linear upper bounds for $p$-subgroups, solvable subgroups, and $p$-solvable subgroups of automorphism groups over fixed finite fields. Motivated by these results, we propose the following conjecture.

\begin{conjecture*}
For every prime power $q$, there exists a constant $C_q>0$ such that, for every function field $F/\F_q$ of genus $g_F\ge2$ with full constant field $\F_q$, one has
\[
   |\Aut(F/\F_q)|\le C_q(g_F-1).
\]
\end{conjecture*}

The paper is organized as follows. Section~2 contains the ramification, Jacobian, restricted ramification tower, and group-theoretic preliminaries. Section~3 proves upper bounds for solvable automorphism groups and $p$-groups. Section~\ref{sec:asymptotic-lower-bounds} proves asymptotic lower bounds using tame towers over the rational function field, unramified towers over Hermitian function fields, and an explicit recursive tower. In Section~\ref{sec:further-applications} we consider prime-to-$p$ and $p$-solvable automorphism groups.
\section{Preliminaries}

Throughout, $F/\F_q$ denotes a function field with full constant field $\F_q$. Then its automorphism group $\Aut(F/\F_q)$ is a finite group.  For $G\le\Aut(F/\F_q)$ we write $F^G$ for the fixed field.  Then $F/F^G$ is Galois and
\[
 \Gal(F/F^G)=G.
\]

We use the following notation throughout.  For a function field $K$ over a finite constant field, let $\mathbb P_K$ denote the set of places of $K$, and let $N_1(K)$ denote the number of places of degree one.  For a finite extension $L/K$, write $\Ram(L/K)$ for the set of places of $K$ ramified in $L$.  If $\ell$ is a prime and $H$ is a pro-$\ell$ group, set
\[
 d(H):=\dim_{\F_\ell}H^1(H,\F_\ell),
\]
so that $d(H)$ is the minimal number of topological generators of $H$.  Finally, $\mu_\ell$ denotes the group of $\ell$-th roots of unity in a fixed separable closure of $K$.

\subsection{Ramification theory}
Let $L/K$ be a finite Galois extension of function fields over $\F_q$. For a place $\mathfrak p$ of $K$, choose a place $\mathfrak P$ of $L$ above $\mathfrak p$.  Since $L/K$ is Galois, the ramification index and different exponent are independent of the choice of $\mathfrak P$; we write
\[
 e_{\mathfrak p}(L/K):=e(\mathfrak P\mid\mathfrak p),
 \qquad
 d_{\mathfrak p}(L/K):=d(\mathfrak P\mid\mathfrak p).
\]
The different divisor is
\[
 \Diff(L/K)=\sum_{\mathfrak p\in\mathbb P_K}
 \sum_{\mathfrak P\mid\mathfrak p}
 d_{\mathfrak p}(L/K)\mathfrak P.
\]
The Hurwitz genus formula may therefore be written in normalized form as
\begin{equation}\label{eq:normalized-Hurwitz}
 \frac{2g_L-2}{[L:K]}
 =2g_K-2+\sum_{\mathfrak p\in\mathbb P_K}
   \frac{d_{\mathfrak p}(L/K)}{e_{\mathfrak p}(L/K)}\deg\mathfrak p,
\end{equation}
see \cite[Theorem~1.3.10]{HX2001}.  If $\mathfrak p$ is tamely ramified, then
\[
 d_{\mathfrak p}(L/K)=e_{\mathfrak p}(L/K)-1
\]
by Dedekind's different theorem \cite[Proposition~1.3.12]{HX2001}. Let $I_j(\mathfrak P)$ be the $j$-th lower ramification group of $\mathfrak P\mid\mathfrak p$.  Hilbert's different formula \cite[Theorem~1.4.10]{HX2001} gives
\begin{equation}\label{eq:Hilbert-different}
 d_{\mathfrak p}(L/K)=\sum_{j\ge0}\bigl(|I_j(\mathfrak P)|-1\bigr).
\end{equation}

\subsection{L-polynomials, Jacobians, and Frobenius}\label{subsec:L-polynomial-prelim}

We first recall some basic facts about the zeta function of a function field. For further details, we refer to \cite[Section~1.6]{HX2001}. 

Let $F/\F_q$ be a function field of genus $g$.  For every integer $m\ge1$, let $N_m(F)$ denote the number of rational places after extending the constant field to $\F_{q^m}$. The zeta function of $F$ is defined by
\[
   Z_F(T)
   =
   \exp\left(
      \sum_{m\ge1}N_m(F)\frac{T^m}{m}
   \right).
\]
We have
\begin{equation}\label{eq:zeta-L-polynomial}
   Z_F(T)
   =
   \frac{L_F(T)}{(1-T)(1-qT)},
\end{equation}
where $L_F(T)\in\Z[T]$ is a polynomial of degree $2g$, called the $L$-polynomial of $F$. Write
\begin{equation}\label{eq:L-roots}
   L_F(T)
   =
   \prod_{i=1}^{2g}(1-\alpha_iT).
\end{equation}
Then \(
   |\alpha_i|=\sqrt q
\) for every $i$ (by the Hasse-Weil theorem), and
\begin{equation}\label{eq:Nm-alpha}
   N_m(F)
   =
   q^m+1-\sum_{i=1}^{2g}\alpha_i^m
   \qquad(m\ge1).
\end{equation}
A  function field \(F/\F_q\)  is called \emph{maximal over \(\F_q\)} if it attains the Hasse-Weil upper bound, that is, \(
   N_1(F)=q+1+2g\sqrt q.
\)

We next introduce some results on the Jacobians of function fields. For more details, we refer to \cite{MilneAbelianVarieties1986,Mumford2008}.

Let $J_F/\F_q$ denote the Jacobian of $F$. Then $J_F$ is an abelian variety of dimension $g$ defined over $\F_q$. Its group of $\F_q$-rational points is canonically isomorphic to the degree-zero divisor class group:
\begin{equation}\label{eq:class-jacobian}
   \operatorname{Cl}^0(F)
   \cong
   J_F(\F_q).
\end{equation}
Let $\pi_F$ denote the $q$-power Frobenius endomorphism of $J_F/\F_q$. For every prime $\ell\ne p$, define
\[
   T_\ell J_F
   =
   \varprojlim_m J_F[\ell^m](\overline{\F}_q),
\]and let  \(
   V_\ell J_F
   =
   T_\ell J_F\otimes_{\Z_\ell}\mathbb{Q}_\ell.
\) Then $T_\ell J_F$ is a free $\Z_\ell$-module of rank $2g$, and $V_\ell J_F$ is a $2g$-dimensional vector space over $\mathbb{Q}_\ell$. The Frobenius endomorphism $\pi_F$ acts naturally on $J_F(\overline{\F}_q)$, and hence on $T_\ell J_F$ and $V_\ell J_F$. The characteristic polynomial of this action is independent of $\ell$ and satisfies
\begin{equation}\label{eq:L-Frobenius}
   \det\bigl(X-\pi_F\mid V_\ell J_F\bigr)
   =
   X^{2g}L_F(X^{-1})
   =
   \prod_{i=1}^{2g}(X-\alpha_i).
\end{equation}
Thus the reciprocal roots $\alpha_i$ of $L_F(T)$ are precisely the eigenvalues of Frobenius on $V_\ell J_F$.

We shall also use two standard facts concerning this action. First, Frobenius acts semisimply on $V_\ell J_F$ \cite[Lemmas~19.2 and~19.3]{MilneAbelianVarieties1986} or \cite[\S21, Proposition on p.~188]{Mumford2008}. Second, the natural map
\[
   \operatorname{End}_{\F_q}(J_F)\otimes_{\Z}\mathbb{Q}_\ell
   \longrightarrow
   \operatorname{End}_{\mathbb{Q}_\ell}(V_\ell J_F)
\]
is injective  \cite[Theorem~12.5]{MilneAbelianVarieties1986}. Consequently, if $\pi_F$ acts on $V_\ell J_F$ as multiplication by an integer $a$, then \(
   \pi_F=[a]
\) as an endomorphism of $J_F/\F_q$. The $\F_q$-rational points of $J_F$ are precisely the points of $J_F(\overline{\F}_q)$ fixed by Frobenius. Hence
\begin{equation}\label{eq:rational-points-Frobenius}
   J_F(\F_q)
   =
   \ker\!\left(
      \pi_F-1:
      J_F(\overline{\F}_q)
      \longrightarrow
      J_F(\overline{\F}_q)
   \right).
\end{equation}
Finally, if $n$ is positive and $p\nmid n$, then by \cite[Theorem~8.2 and Remark~8.4]{MilneAbelianVarieties1986},
\begin{equation}\label{eq:prime-to-p-torsion}
   J_F[n](\overline{\F}_q)
   \cong
   (\Z/n\Z)^{2g}.
\end{equation}

\begin{corollary}\label{cor:extremal-class-group}
Assume that $q=r^2$.  If $F/\F_q$ is maximal over \(\F_q\), then
\begin{equation}\label{eq:maximal-L-class}
   L_F(T)=(1+rT)^{2g}\quad \text{and} \quad
   \operatorname{Cl}^0(F)\cong(\Z/(r+1)\Z)^{2g}.
\end{equation}
\end{corollary}

\begin{proof}
By the Hasse-Weil equality and \eqref{eq:Nm-alpha},
\[
   \sum_{i=1}^{2g}\alpha_i=-2gr.
\]
Since $|\alpha_i|=r$ for every $i$, thus $\alpha_i=-r$ for all $i$. Hence $L_F(T)=(1+rT)^{2g}$. By \eqref{eq:L-Frobenius}, the characteristic polynomial of $\pi_F$ on $V_\ell J_F$ is $(X+r)^{2g}$. Since Frobenius acts semisimply on $V_\ell J_F$, as it divides $(X+r)^{2g}$, its minimal polynomial must be $X+r$. Thus $\pi_F$ acts as the scalar  $-r$ on $V_\ell J_F$. The endomorphism $[-r]$ induces the same scalar action on $V_\ell J_F$. Hence $\pi_F-[-r]$ acts trivially on $V_\ell J_F$, and the injectivity of the endomorphism representation above gives
\[
   \pi_F=[-r]
\]
as an endomorphism of $J_F/\F_q$. Therefore, by \eqref{eq:rational-points-Frobenius},
\[
   J_F(\F_q)=\ker(\pi_F-1)=J_F[r+1](\overline{\F}_q).
\]
Since $p\mid r$, we have $p\nmid (r+1)$, and \eqref{eq:prime-to-p-torsion} yields
\[
   J_F(\F_q)\cong(\Z/(r+1)\Z)^{2g}.
\]
The desired assertion follows from \eqref{eq:class-jacobian}.
\end{proof}

\section{Tamely ramified \texorpdfstring{pro-$\ell$}{pro-l} towers}
\label{subsec:tame-prelim}

Let $K/\F_q$ be a function field, let $\ell\ne p$ be a prime, and let $S,T$ be disjoint finite sets of places of $K$ with $S\ne\varnothing$. Denote by $K_T^S$ the maximal pro-$\ell$ Galois extension of $K$ which is unramified outside $T$ and in which every place of $S$ splits completely, and write
\[
   G_T^S:=\Gal(K_T^S/K).
\]
We recall the class field-theoretic notation that will be used below. Let
\[
   \mathcal O_K(S)
   :=
   \{f\in K : v_Q(f)\ge0
       \text{ for every }Q\notin S\}
\]
be the ring of $S$-integers and let \(
   E_K(S):=\mathcal O_K(S)^\times
\) be its group of $S$-units.  For a place $Q$ of $K$, let
\[
   \kappa(Q):=\mathcal O_Q/\mathfrak m_Q
\]
denote its residue field, and let $\kappa(Q)^\times$ denote its multiplicative group. For a finite set $T$ disjoint from $S$, let \(
   \mathfrak m_T:=\sum_{Q\in T}Q
\) and denote by \(
   \operatorname{Cl}_K(S,T)
\) the ray class group of $\mathcal O_K(S)$ modulo $\mathfrak m_T$. When $T=\varnothing$, we simply write \(
   \operatorname{Cl}_K(S)
   :=
   \operatorname{Cl}_K(S,\varnothing).
\) For a finitely generated abelian group $A$, write
\[
   d_\ell(A):=
   \dim_{\F_\ell}(A/\ell A),
\]
where in multiplicative notation $A/\ell A$ means $A/A^\ell$. If $A$ is finite, then
\[
   d_\ell(A)
   =
   \dim_{\F_\ell}A[\ell].
\]
By global class field theory, the maximal abelian quotient of $G_T^S$ is the $\ell$-primary part of
$\operatorname{Cl}_K(S,T)$. Hence by the Burnside basis theorem,
\begin{equation}\label{eq:generator-ray-class-rank}
   d(G_T^S)
   =
   d_\ell\operatorname{Cl}_K(S,T).
\end{equation}
We use the following consequence of the Golod--Shafarevich criterion \cite[Theorem~7]{AnglesMaire2002}.

\begin{theorem}\label{thm:tame-GS}
Define
\[
 \theta_{K,T}:=
 \begin{cases}
  1,&\mu_\ell\subset K\text{ and }T=\varnothing,\\
  0,&\text{otherwise}.
 \end{cases}
\]
If
\[
   d(G_T^S)\ge 2+2\sqrt{|S|+\theta_{K,T}},
\]
then $G_T^S$ is infinite. In particular, if  $T\ne\varnothing$ and $|S|=1$, then $d(G_T^S)\ge4$ is sufficient.
\end{theorem}

For the rational function field, we shall use the following lower bound for the generator rank.

\begin{lemma}\label{lem:rational-ray-rank}
Let $K=\F_q(x)$, let $S\ne\varnothing$ and $T$ be disjoint finite sets of places of $K$, and suppose that
\[
   \ell\mid (q^{\deg Q}-1)
   \qquad\text{for every }Q\in T.
\]
Set
\[
   \delta_\ell(q):=
   \begin{cases}
      1,&\ell\mid (q-1),\\
      0,&\ell\nmid (q-1).
   \end{cases}
\]
Then
\[
   d(G_T^S)
   =
   d_\ell\operatorname{Cl}_K(S,T)
   \ge
   |T|-|S|+1-\delta_\ell(q).
\]
\end{lemma}

\begin{proof}
If $T=\varnothing$, the asserted lower bound is immediate, so assume that $T\ne\varnothing$. The ray class exact sequence gives
\begin{equation}\label{eq:ray-class-exact-sequence}
\begin{split}
E_K(S)
\xrightarrow{\rho_T}
\prod_{Q\in T}\kappa(Q)^\times  
&\longrightarrow
\operatorname{Cl}_K(S,T)
\longrightarrow
\operatorname{Cl}_K(S)
\longrightarrow1;
\end{split}
\end{equation}
see \cite[Section~1]{AnglesMaire2002}. Set \(
   B_T:=\prod_{Q\in T}\kappa(Q)^\times,
\) By exactness, $B_T/\operatorname{im}(\rho_T)$ embeds into $\operatorname{Cl}_K(S,T)$. Consequently,
\[
d_\ell\operatorname{Cl}_K(S,T)
   \ge d_\ell(B_T)-d_\ell(\operatorname{im}(\rho_T))
   \ge d_\ell(B_T)-d_\ell(E_K(S)).
\]
By Dirichlet's unit theorem, \(
   E_K(S)
   \cong
   \F_q^\times\times\mathbb Z^{|S|-1}.
\) It follows that
\[
   d_\ell(E_K(S))
   =
   |S|-1+\delta_\ell(q).
\]
On the other hand, \(
   \kappa(Q)^\times
   \cong
   \F_{q^{\deg Q}}^\times
\) is cyclic of order $q^{\deg Q}-1$. By hypothesis, $\ell\mid (q^{\deg Q}-1$), and hence $  d_\ell\bigl(\kappa(Q)^\times\bigr)=1$ for $Q \in T$. Therefore
\[
   d_\ell(B_T)=|T|,
\]
and thus \(
   d_\ell\operatorname{Cl}_K(S,T)
   \ge
   |T|-|S|+1-\delta_\ell(q).
\) Finally, \eqref{eq:generator-ray-class-rank} gives \(
   d(G_T^S)
   =
   d_\ell\operatorname{Cl}_K(S,T),
\) which completes the proof.
\end{proof}

\subsection{Chief series of finite solvable groups}
We use standard facts on chief series from \cite[Section~1.8]{KurzweilStellmacher2004}.

\begin{definition}
A \emph{chief series} of a finite group $G$ is a normal series
\[
 G=G_0\supsetneq G_1\supsetneq\cdots\supsetneq G_s=1
\]
such that $G_i\trianglelefteq G$ for every $i$ and there is no $M\trianglelefteq G$ with
\[
 G_{i+1}\subsetneq M\subsetneq G_i.
\]
The quotients $G_i/G_{i+1}$ are the \emph{chief factor} of $G$.
\end{definition}

\begin{proposition}\label{prop:chief-factor}
Let $G$ be finite and solvable.  For every chief series
\[
 G=G_0\supsetneq G_1\supsetneq\cdots\supsetneq G_s=1
\]
and every $i$, there are a prime $\ell_i$ and $a_i\ge1$ such that
\[
 G_i/G_{i+1}\cong(\Z/\ell_i\Z)^{a_i}.
\]
\end{proposition}

\begin{proof}
This is \cite[Theorem~6.1.5]{KurzweilStellmacher2004}.
\end{proof}

If $F/F_0$ is Galois with solvable group $G$ and $F_i:=F^{G_i}$, then
\begin{equation}\label{eq:chief-fixed-fields}
 \Gal(F_{i+1}/F_i)\cong G_i/G_{i+1}
 \cong(\Z/\ell_i\Z)^{a_i}.
\end{equation}

\subsection{Finite subgroups of \texorpdfstring{$\PGL_2(\F_q)$}{PGL2(Fq)}}

We record two consequences of the subgroup structure of $\PGL_2(\F_q)$ that will be used below.  The first concerns elementary abelian subgroups.

\begin{proposition}\label{prop:PGL2-elementary}
Let
\[
   K=\F_q(x),\qquad q=p^e,
\]
where $p=\charac(\F_q)$. Then
\[
   \Aut(K/\F_q)\cong \PGL_2(\F_q).
\]
Let $\ell$ be a prime. If $A$ is an elementary abelian $\ell$-subgroup of $\Aut(K/\F_q)$, then
\[
   |A|\leq m_\ell(q),
\]
where
\[
   m_\ell(q):=
   \begin{cases}
      q, & \ell=p,\\
      4, & \ell=2\neq p,\\
      \ell, & \ell\neq p,\ \ell\text{ odd},\ \ell\mid(q^2-1),\\
      1, & \ell\neq p,\ \ell\nmid(q^2-1).
   \end{cases}
\]
\end{proposition}

\begin{proof}
Every $\F_q$-automorphism of $\F_q(x)$ is a fractional linear transformation \cite[Theorem~11.14]{HirschfeldKorchmarosTorres}. Hence
\[
   \Aut(\F_q(x)/\F_q)\cong \PGL_2(\F_q).
\]

We now apply the complete list of subgroups of $\PGL_2(\F_q)$ given in \cite[Appendix~A, Theorem~A.8]{HirschfeldKorchmarosTorres}. 

First suppose that $\ell=p$.  In case~(ii) of \cite[Appendix~A, Theorem~A.8]{HirschfeldKorchmarosTorres}, an elementary abelian $p$-subgroup has order $p^f$ with $f\le e$. Therefore
\[
   |A|=p^f\le p^e=q.
\]
This bound is attained by the translation subgroup
\[
   \{x\mapsto x+a:a\in\F_q\}
   \cong(\F_q,+),
\]
which is elementary abelian of order $q$.

Suppose next that $\ell\neq p$ is odd.  Since $A$ itself is an elementary abelian $\ell$-group, we may apply
\cite[Appendix~A, Theorem~A.8]{HirschfeldKorchmarosTorres} directly to $A$.  Among the groups in that list, the only one which is an elementary abelian group of odd order prime to $p$ is the cyclic group in case~(i).  Hence $A$ is cyclic.  Since $A$ is elementary abelian, either $A=1$ or
\[
   A\cong C_\ell.
\]
By case~(i) of Theorem~A.8, a cyclic subgroup of order $\ell$ can occur only when
\[
   \ell\mid(q-1)
   \qquad\text{or}\qquad
   \ell\mid(q+1),
\]
that is, precisely when $\ell\mid(q^2-1)$.  Conversely, when $\ell\mid(q^2-1)$, case~(i) gives a cyclic subgroup of order $\ell$.  Thus
\[
   |A|\le
   \begin{cases}
      \ell,&\ell\mid(q^2-1),\\
      1,&\ell\nmid(q^2-1).
   \end{cases}
\]

Finally suppose that $\ell=2\neq p$, so that $q$ is odd. Again apply Theorem~A.8 directly to the elementary abelian $2$-group $A$.  The only groups in the list which can themselves be elementary abelian $2$-groups are the cyclic group $C_2$ in case~(i) and the dihedral group $D_4$ in case~(iii).  In the notation of Theorem~A.8, $D_{2n}$ has order $2n$, and hence
\[
   D_4\cong C_2\times C_2.
\]
Consequently $|A|\le4$.  Since $q$ is odd, $2\mid(q-1)$, and case~(iii) of Theorem~A.8 gives such a subgroup $D_4$. Therefore the bound $4$ is attained.
\end{proof}

\begin{proposition}
\label{prop:PGL2-solvable}
Let $K=\F_q(x)$ and let $G\leq\Aut(K/\F_q)$ be a  solvable subgroup. Then
\[
   |G|\leq R_q,
\]
where
\begin{equation}\label{eq:Rq-definition}
 R_q=
 \begin{cases}
  6,&q=2,\\
  24,&q=3,5,\\
  q(q-1),&\text{otherwise}.
 \end{cases}
\end{equation}
\end{proposition}

\begin{proof}
By \cite[Appendix~A, Theorem~A.8]{HirschfeldKorchmarosTorres}, a solvable subgroup \(G\le \PGL_2(\F_q)\), with \(q=p^e\), is of one of the following types: cyclic or dihedral, an elementary abelian \(p\)-group, a group \((C_{p})^h\rtimes C_n\) with \(h\le e\) and \(n\mid(q-1)\), one of \(A_4,S_4\), or a solvable subfield subgroup.  The case \(A_5\) is non-solvable. The corresponding orders satisfy
\[
 |G|\le q+1,\qquad 2(q+1),\qquad  q,\qquad
 q(q-1),\qquad 24,
\]
respectively.  Moreover, by \cite[Appendix~A, Example~A.7(5)]{HirschfeldKorchmarosTorres}, \(\PSL_2(\F_{p^f})\) is simple for \(p^f\ne2,3\), hence a solvable subfield subgroup occurring in cases~(viii) or~(ix) of Theorem~A.8 has order at most \(24\).  Therefore
\[
 |G|\le \max\{q(q-1),\,2(q+1),\,24\}.
\]
If \(q\ge7\), the maximum is \(q(q-1)\), and this value is attained by the affine subgroup
\[
 \{x\mapsto ax+b:\ a\in\F_q^\times,\ b\in\F_q\},
\]
of order \(q(q-1)\). It remains to consider \(q=2,3,4,5\).  We have
\[
 \PGL_2(\F_2)\cong S_3,\qquad
 \PGL_2(\F_3)\cong S_4,
\]
giving maxima \(6\) and \(24\), respectively.  For \(q=4\), Theorem~A.8(v) does not occur since \(p=2\), while all the remaining solvable cases have order at most \(12;\)  the affine subgroup has order \(12\).  For \(q=5\), Theorem~A.8(v) gives an \(S_4\)-subgroup, so the maximum is \(24\).
\end{proof}

\section{Upper bounds for solvable automorphism groups}
\subsection{Upper bounds for solvable subgroups}
Let $F/\F_q$ have genus $g_F\ge2$, let $G\le\Aut(F/\F_q)$ be solvable, and write
\[
 F_0:=F^G.
\]
Then $F/F_0$ is Galois with group $G$, and
\begin{equation}\label{eq:solvable-Hurwitz}
 \frac{2g_F-2}{|G|}
 =2g_{F_0}-2+\sum_{\mathfrak p\in\mathbb P_{F_0}}
 \frac{d_{\mathfrak p}(F/F_0)}{e_{\mathfrak p}(F/F_0)}\deg\mathfrak p.
\end{equation}

\subsubsection{The case $g_{F_0}\ge1$}

\begin{proposition}\label{prop:g0-ge2}
If $g_{F_0}\ge2$, then
\[
 |G|\le g_F-1.
\]
\end{proposition}

\begin{proof}
The right-hand side of \eqref{eq:solvable-Hurwitz} is at least $2$.
\end{proof}

\begin{proposition}\label{prop:g0-eq1}
Assume $g_{F_0}=1$, and let $r$ be the smallest prime divisor of $|G|$.  Then
\[
 |G|\le\frac{2r}{r-1}(g_F-1)\le4(g_F-1).
\]
\end{proposition}

\begin{proof}
The extension $F/F_0$ is ramified, otherwise Hurwitz would give $g_F=1$. Choose a ramified place $\mathfrak p$ of $F_0$.  Since $e_{\mathfrak p}(F/F_0)\ge r$ and $d_{\mathfrak p}(F/F_0)\ge e_{\mathfrak p}(F/F_0)-1$, we have
\[
 \frac{d_{\mathfrak p}(F/F_0)}{e_{\mathfrak p}(F/F_0)}
 \deg\mathfrak p\ge\frac{r-1}{r}.
\]
Hence \eqref{eq:solvable-Hurwitz} gives
\[
 \frac{2g_F-2}{|G|}\ge\frac{r-1}{r},
\]
which proves the claim.
\end{proof}

\subsubsection{The case $g_{F_0}=0$}
Let $R_q$ denote the constant from Proposition~\ref{prop:PGL2-solvable}.  Throughout this subsection, $r$ denotes the smallest prime divisor of $|G|$.  Fix a chief series
\[
   G=G_0\triangleright G_1\triangleright\cdots\triangleright G_n=1,
\]
and let \(
   F_i:=F^{G_i}
\) be the corresponding fixed fields.  Recall that, since $G$ is solvable, every chief factor $G_i/G_{i+1}$ is an elementary abelian group.  Let $i$ be the first index for which the genus becomes positive, that is,
\[
   g_{F_i}=0
   \qquad\text{and}\qquad
   g_{F_{i+1}}>0.
\]
For this index, there is a prime $\ell$ and an integer $a\ge1$ such that
\[
   G_i/G_{i+1}
   \cong
   \Gal(F_{i+1}/F_i)
   \cong
   (\Z/\ell\Z)^a.
\]
Thus $\ell$ is the prime attached to the first chief factor at which positive genus appears. Since $\ell\mid |G_i/G_{i+1}|$ and hence $\ell\mid |G|$, we have \(
   r\le \ell.
\)

\begin{lemma}\label{lem:prime-layer-large}
Let $L/K$ be cyclic of prime degree $\ell\ge5$.  If $g_K=0$ and $g_L>0$, then
\[
 2g_L-2\ge\ell-3.
\]
In particular, $g_L\ge2$.
\end{lemma}

\begin{proof}
The Hurwitz genus formula gives
\[
 2g_L-2=-2\ell+\deg\Diff(L/K).
\]
If $\ell\ne\charac(\F_q)$, the extension is tame and every ramified place is totally ramified.  Thus
\[
 \deg\Diff(L/K)=(\ell-1)
 \sum_{\mathfrak p\in\Ram(L/K)}\deg\mathfrak p.
\]
Since $g_L>0$, the sum of the degrees of the ramified places is at least $3$, hence $2g_L-2\ge-2\ell+3(\ell-1)=\ell-3$. Now let $\ell=\charac(\F_q)$.  Write the Artin--Schreier extension in reduced form as $y^\ell-y=f$.  If $m_{\mathfrak p}>0$ is the pole order of $f$ at $\mathfrak p$, then $\ell\nmid m_{\mathfrak p}$ and
\[
 d_{\mathfrak p}(L/K)=(\ell-1)(m_{\mathfrak p}+1),
\]
see \cite[Proposition 3.7.8]{Stichtenoth2009}. Therefore
\[
 \deg\Diff(L/K)
 =(\ell-1)\sum_{\mathfrak p}(m_{\mathfrak p}+1)\deg\mathfrak p.
\]
The last sum is an integer strictly larger than $2$, so it is at least $3$, and the same inequality follows.
\end{proof}

\begin{theorem}\label{thm:precise-rational}
Assume $g_{F_0}=0$.  Choose a chief series
\[
 G=G_0\supsetneq G_1\supsetneq\cdots\supsetneq G_s=1,
 \qquad F_i:=F^{G_i},
\]
and let $i$ be maximal such that $g_{F_i}=0$.  Write
\[
 \Gal(F_{i+1}/F_i)\cong(\Z/\ell\Z)^a.
\]
Define
\[
k_{\ell,r}:=
\begin{cases}
\dfrac{2r\ell}{r-1}, & \ell\in\{2,3\},\\[6pt]
\dfrac{2\ell}{\ell-3}, & \ell\ge 5.
\end{cases}
\]
Then
\[
|G|\le R_q\,k_{\ell,r}\,m_\ell(q)(g_F-1),
\]
where $R_q$ and $m_\ell(q)$ are defined in \eqref{eq:Rq-definition} and  \Cref{prop:PGL2-elementary}, respectively.
\end{theorem}

\begin{proof}
Set
\[
 A:=\Gal(F_{i+1}/F_i)\cong(\Z/\ell\Z)^a
\]
and choose a chain of subgroups
\[
 A=A_0\supsetneq A_1\supsetneq\cdots\supsetneq A_a=1,
 \qquad [A_j:A_{j+1}]=\ell.
\]
For $K_j:=F_{i+1}^{A_j}$ we have
\[
 F_i=K_0\subsetneq K_1\subsetneq\cdots\subsetneq K_a=F_{i+1},
 \qquad [K_{j+1}:K_j]=\ell.
\]
Choose $j$ maximal with $g_{K_j}=0$ and set
\[
 K:=K_j,\qquad L:=K_{j+1}.
\]
Then
\[
 g_K=0<g_L,\qquad [L:K]=\ell.
\]

Since $A$ is abelian, $K/F_i$ is Galois with
\[
 \Gal(K/F_i)\cong A/A_j,
\]
so Proposition~\ref{prop:PGL2-elementary} gives
\[
 [K:F_i]\le m_\ell(q).
\]
Also
\[
 \Gal(F_i/F_0)\cong G/G_i.
\]
Both $F_i$ and $F_0$ are rational over $\F_q$, and hence
\[
 \Gal(F_i/F_0)\le\Aut(F_i/\F_q)\cong\PGL_2(\F_q).
\]
By Proposition~\ref{prop:PGL2-solvable},
\[
 [F_i:F_0]\le R_q,
 \qquad [K:F_0]\le m_\ell(q)\,R_q.
\]

If $\ell\ge5$, Lemma~\ref{lem:prime-layer-large} gives $2g_L-2\ge\ell-3$.  Applying the Hurwitz genus formula to $F/L$ gives
\[
 [F:L]\le\frac{2(g_F-1)}{\ell-3}.
\]
Therefore
\[
 |G|=[F:L][L:K][K:F_i][F_i:F_0]
 \le\frac{2\ell}{\ell-3}m_\ell(q)\,R_q(g_F-1).
\]

Let $\ell\in\{2,3\}$.  If $g_L\ge2$, then $[F:L]\le g_F-1$.  If $g_L=1$, let $\widetilde A_{j+1}\le G_i$ be the inverse image of $A_{j+1}\le G_i/G_{i+1}=A$.  Then $L=F^{\widetilde A_{j+1}}$ and $F/L$ is Galois with group $\widetilde A_{j+1}$.  If $r_L$ is the smallest prime divisor of $[F:L]$, then $r_L\ge r$.  Since $g_L=1<g_F$, the extension $F/L$ is ramified.  Choosing a ramified place $\mathfrak p$ of $L$ and using $d_{\mathfrak p}(F/L)\ge e_{\mathfrak p}(F/L)-1$ gives
\[
 \frac{2g_F-2}{[F:L]}
 \ge \frac{r_L-1}{r_L}.
\]
Thus
\[
 [F:L]\le\frac{2r_L}{r_L-1}(g_F-1)
 \le\frac{2r}{r-1}(g_F-1).
\]
This also dominates the estimate for $g_L\ge2$.  Hence
\[
 |G|\le\frac{2r\ell}{r-1}m_\ell(q)\,R_q(g_F-1).
\]
\end{proof}

\begin{corollary}\label{cor:rational-exact}
Assume $g_{F_0}=0$. Define
\[
D_q:=
\begin{cases}
\max\{36,8q\}, & p=2,\\[4pt]
12q, & p=3,\\[4pt]
\displaystyle
\max\left\{36,\frac{2pq}{p-3}\right\}, & p\ge 5.
\end{cases}
\]
Then
\[
|G|\le R_q D_q(g_F-1).
\]
\end{corollary}

\begin{proof}
By Theorem~\ref{thm:precise-rational},
\[
|G|\le R_q\,k_{\ell,r}\,m_\ell(q)(g_F-1).
\]
where $r$ is the smallest prime divisor of $|G|$.

For $\ell=2$ one has $r=2$, and hence
\[
 8m_2(q)=
 \begin{cases}
  8q,&p=2,\\
  32,&p\ne2.
 \end{cases}
\]
For $\ell=3$ one has $r\in\{2,3\}$, so
\[
 \frac{2r\ell}{r-1}m_3(q)\le 12m_3(q)
 =
 \begin{cases}
  12q,&p=3,\\
  36,&p\ne3.
 \end{cases}
\]
It remains to consider $\ell\ge5$.  If $\ell=p$, then necessarily $p\ge5$ and
\[
 \frac{2\ell}{\ell-3}m_\ell(q)
 =
 \frac{2pq}{p-3}.
\]
Suppose that $\ell\ne p$.  If $\ell\nmid(q^2-1)$, then $m_\ell(q)=1$ and
\[
 \frac{2\ell}{\ell-3}\le5.
\]
If $\ell\mid(q^2-1)$, then $m_\ell(q)=\ell$ and the contribution is
\[
 \frac{2\ell^2}{\ell-3}.
\]
For $\ell=5$ this equals $25$.  Assume therefore that $\ell\ge7$.

If $p=2$, then $\ell\le q+1$, and
\[
 \frac{2\ell^2}{\ell-3}
 \le 8(\ell-1)
 \le 8q.
\]
If $p$ is odd, then $\ell$ is odd and $\ell\mid(q-1)$ or $\ell\mid(q+1)$.  Since $q\pm1$ is even,
\[
 2\ell\le q+1.
\]
Moreover,
\[
 \frac{2\ell^2}{\ell-3}\le4\ell-2\le2q.
\]
Thus this term is bounded by $12q$ when $p=3$, and by $2pq/(p-3)$ when $p\ge5$.  Combining the above cases gives the stated bound.
\end{proof}

\begin{theorem}\label{cor:Sq-upper}
Let $q=p^e$.  Define
\[
 C_q^{(\mathrm{sol})}:=
 R_qD_q
\]
where $R_q$ and $D_q$ are defined in \eqref{eq:Rq-definition} and  \Cref{cor:rational-exact}, respectively. Then for every function field $F/\F_q$ of genus $g_F\ge2$ and every  solvable subgroup $G\le\Aut(F/\F_q)$,
\[
 |G|\le C_q^{(\mathrm{sol})}(g_F-1).
\]
Moreover, for every $q\ge4$,
\[
 C_q^{(\mathrm{sol})}\le 12q^2(q-1).
\]
In particular, \(
 C_2^{(\mathrm{sol})}=216,
 C_3^{(\mathrm{sol})}=864.
\)
\end{theorem}

\begin{proof}
Write $F_0=F^G$. If $g_{F_0}=0$, Corollary~\ref{cor:rational-exact} gives the required bound.  If $g_{F_0}=1$, Proposition~\ref{prop:g0-eq1} gives
\[
 |G|\le4(g_F-1),
\]
while if $g_{F_0}\ge2$, Proposition~\ref{prop:g0-ge2} gives
\[
 |G|\le g_F-1.
\]
Since $C_q^{(\mathrm{sol})}\ge4$, the same bound holds in all cases.
\end{proof}

\subsection{Upper bounds for \texorpdfstring{$p$-subgroups}{p-subgroups}}
\label{sec:p-subgroup-upper-bounds}

Let $q=p^e$, where $p=\charac(\F_q)$. Define
\begin{equation}\label{eq:def-Cqp}
C_q^{(p)}:=
\begin{cases}
   4q,&p=2,\\[2mm]
   \dfrac92q,&p=3,\\[2mm]
   \dfrac{2pq}{p-3},&p\ge5.
\end{cases}
\end{equation}
We prove the bound announced in \eqref{eq:intro-uniform-p}.  The argument uses two features which are specific to $p$-groups: their chief factors have order $p$, and wild ramification gives a stronger lower bound for the different.

\begin{lemma}\label{lem:p-wild-different}
Let $L/K$ be a finite Galois $p$-extension in characteristic $p$.  If a place $\mathfrak p$ of $K$ is ramified in $L/K$, then
\[
 d_{\mathfrak p}(L/K)\ge2\bigl(e_{\mathfrak p}(L/K)-1\bigr),
\]
and therefore
\[
 \frac{d_{\mathfrak p}(L/K)}{e_{\mathfrak p}(L/K)}
 \ge\frac{2(p-1)}p.
\]
\end{lemma}

\begin{proof}
Let $\mathfrak P$ be a place of $L$ above $\mathfrak p$.  Since the inertia group is a $p$-group, its first two lower ramification groups satisfy $I_0(\mathfrak P)=I_1(\mathfrak P)$  \cite[Theorem~1.4.7]{HX2001}.  Hilbert's different formula \eqref{eq:Hilbert-different} gives
\[
 d_{\mathfrak p}(L/K)
 \ge2\bigl(|I_0(\mathfrak P)|-1\bigr)
 =2\bigl(e_{\mathfrak p}(L/K)-1\bigr).
\]
Since a ramified $p$-extension has $e_{\mathfrak p}(L/K)\ge p$, the second inequality follows.
\end{proof}

\begin{proposition}\label{prop:pgroup-genus-one}
Let $P\le\Aut(F/\F_q)$ be a finite $p$-subgroup and put $F_0=F^P$.  If $g_{F_0}=1$, then
\[
 |P|\le\frac{p}{p-1}(g_F-1).
\]
\end{proposition}

\begin{proof}
The extension $F/F_0$ is ramified, since an unramified extension of a genus-one function field again has genus one.  Choose a ramified place $\mathfrak p$ of $F_0$.  Lemma~\ref{lem:p-wild-different} and \eqref{eq:solvable-Hurwitz} give
\[
 \frac{2g_F-2}{|P|}
 \ge\frac{2(p-1)}p,
\]
which proves the assertion.
\end{proof}

\begin{proposition}\label{prop:pgroup-rational}
Let $P\le\Aut(F/\F_q)$ be a  $p$-subgroup and put $F_0=F^P$.  If $g_{F_0}=0$, then
\[
 |P|\le
 \begin{cases}
  4q(g_F-1),&p=2,\\[2pt]
  \dfrac92q(g_F-1),&p=3,\\[6pt]
  \dfrac{2pq}{p-3}(g_F-1),&p\ge5.
 \end{cases}
\]
\end{proposition}

\begin{proof}
Choose a chief series
\[
 P=P_0\supsetneq P_1\supsetneq\cdots\supsetneq P_m=1,
 \qquad F_i:=F^{P_i}.
\]
Every chief factor of a  $p$-group has order $p$, so $[F_{i+1}:F_i]=p$ at each step of this chief series.  Choose $j$ maximal with $g_{F_j}=0$ and write
\[
 K:=F_j,\qquad L:=F_{j+1},\qquad \nu:=[F:L].
\]
Then
\[
 g_K=0<g_L,\qquad [L:K]=p.
\]
Both $K$ and $F_0$ have full constant field $\F_q$.  Since $g_K=0,$ $K$ is a rational function field. Therefore $\Gal(K/F_0)$ is a $p$-subgroup of $\PGL_2(\F_q)$.  As the $p$-part of $|\PGL_2(\F_q)|$ equals $q$, we obtain
\[
 [K:F_0]\le q.
\]
Consequently
\begin{equation}\label{eq:pgroup-factorization}
 |P|=\nu p[K:F_0].
\end{equation}
Assume first that $p\ge5$.  Lemma~\ref{lem:prime-layer-large} gives $2g_L-2\ge p-3$.  Applying the Hurwitz genus formula to $F/L$ yields
\[
 \nu\le\frac{2(g_F-1)}{p-3}.
\]
Together with \eqref{eq:pgroup-factorization}, this gives
\[
 |P|\le\frac{2pq}{p-3}(g_F-1).
\]

Now let $p\in\{2,3\}$.  If $g_L\ge2$, then Hurwitz gives $\nu\le g_F-1$.  If $g_L=1$, the Galois $p$-extension $F/L$ is ramified, and Lemma~\ref{lem:p-wild-different} gives
\[
 2g_F-2=\deg\Diff(F/L)
 \ge \nu\frac{2(p-1)}p.
\]
Hence
\[
 \nu\le\frac{p}{p-1}(g_F-1).
\]
Using \eqref{eq:pgroup-factorization} in both cases gives
\[
 |P|\le\frac{p^2q}{p-1}(g_F-1),
\]
which is $4q(g_F-1)$ for $p=2$ and $9q(g_F-1)/2$ for $p=3$.
\end{proof}

\begin{theorem}\label{thm:Pq-upper}
Let $q=p^e$, let $F/\F_q$ be a function field of genus $g_F\ge2$, and let
$P\le\Aut(F/\F_q)$ be a  $p$-subgroup.  Then
\[
 |P|\le C_q^{(p)}(g_F-1),
\]
where $C_q^{(p)}$ is defined by \eqref{eq:def-Cqp}.  In particular, one may uniformly take $C_q^{(p)}=5q$ for every $q$.
\end{theorem}

\begin{proof}
Put $F_0=F^P$. If $g_{F_0}=0$, Proposition~\ref{prop:pgroup-rational} gives the required bound.  If $g_{F_0}=1$, apply Proposition~\ref{prop:pgroup-genus-one}; if $g_{F_0}\ge2$, Proposition~\ref{prop:g0-ge2} gives $|P|\le g_F-1$.  Since $C_q^{(p)}\ge1$, the stated estimate holds in all cases.
\end{proof}

\begin{remark}\label{rem:p-versus-solvable}
The constants in the solvable and $p$-group bounds behave quite differently as $q$ grows.  More precisely,
\[
   C_q^{(\mathrm{sol})}=O(q^3),
   \qquad
   C_q^{(p)}=O(q),
\]
so the latter is smaller by a factor of order $q^2$. The main reason is as follows:  In the general solvable case, the genus-zero part may contribute a solvable subgroup of $\PGL_2(\F_q)$ of order of size $q^2$, and the first positive-genus chief factor may introduce an additional factor of order $q$ through the bound for elementary abelian subgroups.  In the $p$-group case, the corresponding subgroup of $\PGL_2(\F_q)$ has order at most $q$, while every chief factor has order exactly $p$.  These two restrictions account for the improvement from order $q^3$ to order $q$.
\end{remark}

\section{Asymptotic lower bounds for solvable automorphism groups}
\label{sec:asymptotic-lower-bounds}

For a function field \(K/\F_q\), a prime \(\ell\ne p\), and disjoint finite sets \(S,T\) of places of \(K\), let \(K_T^S\) denote the maximal pro-\(\ell\) extension of \(K\) in which every place of \(S\) splits completely and which is unramified outside \(T\). We write
\[
   G_T^S:=\Gal(K_T^S/K).
\]
Thus \(S\) is the prescribed splitting set, whereas \(T\) is the allowed ramification set. Since \(\ell\ne p\), every finite subextension of \(K_T^S/K\) is tamely ramified. 

\begin{proposition}\label{prop:restricted-lifting}
Let \(K/\F_q\) be a function field of genus \(g_K\), let \(\ell\ne p\) be a prime, and let \(S,T\) be disjoint finite sets of places with \(S\ne\varnothing\). Assume that the following conditions hold:
\begin{enumerate}
\item \(G_T^S\) is infinite, and \(S\) contains a place \(P\) such that \(\gcd(\deg P,\ell)=1\),
\item There exists a  solvable subgroup \(H\le \Aut(K/\F_q)\) such that \(\sigma(S)=S\) and \(\sigma(T)=T\) for all \(\sigma\in H\), and \(2g_K-2+\sum_{Q\in T}\deg Q>0\).
\end{enumerate}

Then there is a sequence of function fields $E_n/\F_q$ with $g_{E_n}\to\infty$, together with  solvable subgroups $G_n\le\Aut(E_n/\F_q)$, such that
\[
   \frac{|G_n|}{g_{E_n}-1}
   \ge
   \frac{2|H|}
   {\,2g_K-2+\displaystyle\sum_{Q\in T}\deg Q\,}
\]
for every $n$.  In particular, if  $T=\varnothing$ and $g_K\ge2$, the right-hand side is $|H|/(g_K-1)$.
\end{proposition}

\begin{proof}

Fix \(\sigma\in H\), and choose an extension \(\widetilde{\sigma}\) of \(\sigma\) to \(K^{\mathrm{sep}}\). Since \(\sigma(S)=S\) and \(\sigma(T)=T\), conjugation by \(\widetilde{\sigma}\) preserves the splitting and ramification conditions defining \(K_T^S\): indeed, ramification indices and residue degrees are invariant under conjugation; see \cite[Lemma~3.5.2]{Stichtenoth2009}. Hence \(\widetilde{\sigma}(K_T^S)/K\) is again a pro-\(\ell\) extension in which every place of \(S\) splits completely and which is unramified outside \(T\). By the maximality of \(K_T^S\),  \(\widetilde{\sigma}(K_T^S)\subseteq K_T^S.\)  Since \(\widetilde{\sigma}^{-1}\) extends \(\sigma^{-1}\), and \(\sigma^{-1}(S)=S\) and \(\sigma^{-1}(T)=T\), the same argument gives
\(\widetilde{\sigma}^{-1}(K_T^S)\subseteq K_T^S\). Applying \(\widetilde{\sigma}\) to the latter inclusion yields \(K_T^S\subseteq\widetilde{\sigma}(K_T^S)\). Hence
\[
   \widetilde{\sigma}(K_T^S)=K_T^S.
\]

We now show that \(K_T^S/K^H\) is Galois. Let \(\tau:K_T^S\hookrightarrow K^{\mathrm{sep}}\) be a \(K^H\)-embedding. Since \(K/K^H\) is Galois with Galois group \(H\), the restriction \(\tau|_K=\sigma \) for some \(\sigma\in H\). Choose an extension \(\widetilde{\sigma}\) as above, so that \(\widetilde{\sigma}(K_T^S)=K_T^S\). Then \(\widetilde{\sigma}^{-1}\tau\) fixes \(K\).  Thus $\widetilde{\sigma}^{-1}\tau:K_T^S\hookrightarrow K^{\mathrm{sep}}$ is a \(K\)-embedding. Since \(K_T^S/K\) is Galois, it follows that \(
   \widetilde{\sigma}^{-1}\tau(K_T^S)=K_T^S.
\) Consequently \(\tau(K_T^S)=K_T^S\). Thus \(K_T^S/K^H\) is normal, and hence Galois. Restriction to \(K\) therefore gives an exact sequence
\[
   1\longrightarrow G_T^S
   \longrightarrow \Gal(K_T^S/K^H)
   \longrightarrow H
   \longrightarrow 1.
\]

Choose a sequence of finite Galois subextensions \(L_n/K\) of \(K_T^S/K\) such that \([L_n:K]\to\infty\) as \(n\to\infty\). For each \(n\), let \(E_n\) be the Galois closure of \(L_n/K^H\) inside \(K_T^S\). Equivalently, \(E_n\) is the compositum of all \(K^H\)-conjugates of \(L_n\). Since \(K_T^S/K^H\) is Galois, all these conjugate fields are contained in \(K_T^S\), and hence \(E_n\subseteq K_T^S\).

By construction, \(E_n/K^H\) is a finite Galois extension. Moreover, \(L_n\subseteq E_n\), and therefore \([E_n:K]\ge [L_n:K]\to\infty\) as \(n\to\infty\). Restriction to \(K\) gives an exact sequence
\[
   1\longrightarrow \Gal(E_n/K)
   \longrightarrow \Gal(E_n/K^H)
   \longrightarrow H
   \longrightarrow 1.
\]
Since \(E_n\subseteq K_T^S\) and \(K_T^S/K\) is a pro-\(\ell\) extension, \(\Gal(E_n/K)\) is an \(\ell\)-group, hence solvable. As \(H\) is solvable, it follows that \(\Gal(E_n/K^H)\) is solvable. Finally,
\[
   |\Gal(E_n/K^H)|=|H|\,[E_n:K].
\]

We next show that every \(E_n\) has full constant field \(\F_q\). Suppose that its full constant field is \(\F_{q^m}\). Since \(K\F_{q^m}\subseteq E_n\), we have \(m\mid [E_n:K]\), so \(m\) is an \(\ell\)-power. If \(P'\) is a place of \(E_n\) above \(P\), then \(P\) splits completely in \(E_n/K\), and hence \(\deg P'=\deg P\). On the other hand, \(m\mid\deg P'\) because the full constant field of \(E_n\) is \(\F_{q^m}\). Thus \(m\mid\deg P\). Since \(\gcd(\deg P,\ell)=1\), we obtain \(m=1\). Hence \(E_n\) has full constant field \(\F_q\).

Moreover, \(g_{E_n}\to\infty\). Indeed, complete splitting of \(P\) gives \([E_n:K]\) distinct places of degree \(\deg P\) in \(E_n\), so \(\#E_n(\F_{q^{\deg P}})\ge (\deg P) [E_n:K]\). By the Hasse--Weil bound,
\[
    (\deg P) [E_n:K]
   \le
   q^{\deg P}+1+2g_{E_n}q^{\deg P/2},
\]
and therefore \(g_{E_n}\to\infty\) as \([E_n:K]\to\infty\).

Finally, \(E_n/K\) is tamely ramified and unramified outside \(T\). By the Riemann--Hurwitz formula,
\[
   \frac{2g_{E_n}-2}{[E_n:K]}
   =
   2g_K-2+
   \sum_{\substack{Q\in T\\ Q\ {\rm ramified}}}
   \left(1-\frac1{e_Q}\right)\deg Q
   \le
   2g_K-2+\sum_{Q\in T}\deg Q.
\]
Hence
\[
   \frac{|\Gal(E_n/K^H)|}{g_{E_n}-1}
   \ge
   \frac{2|H|}
   {\,2g_K-2+\displaystyle\sum_{Q\in T}\deg Q\,}.
\]
Since $g_{E_n}\to\infty$ and $\Gal(E_n/K^H)\le\Aut(E_n/\F_q)$ is solvable, the sequence
\[
   G_n:=\Gal(E_n/K^H)
\]
has all the asserted properties.  The final statement follows by taking $T=\varnothing$.
\end{proof}

\begin{remark}
Proposition~\ref{prop:restricted-lifting} may be viewed as a class field tower version of the Galois-closure construction in \cite[Section 2]{GarciaStichtenoth2007}. 
\end{remark}

\subsection{Tamely ramified towers over the rational function field}

We first give a uniform construction for $q\ge5$, and then treat the small constant fields $\F_q$ with  $q=2,3,4$.

\begin{theorem}\label{thm:tame-solvable-lower-general}
For every prime power $q\ge5$, there is a sequence of function fields $F_i/\F_q$ with $g_{F_i}\to\infty$ and finite solvable subgroups $G_i\le\Aut(F_i/\F_q)$ such that
\[
   |G_i|\ge \frac{2q(q-1)}{q-2}(g_{F_i}-1)
\]
for every $i$.
\end{theorem}

\begin{proof}
Let $K=\F_q(x)$. Then $g_K=0$. Let $P_\infty$ be the infinite place of $K$, let $S=\{P_\infty\}$, and let
\[
   H=\{x\mapsto ax+b:a\in\F_q^\times,\ b\in\F_q\}.
\]
Then $H$ is solvable of order $q(q-1)$ and fixes $P_\infty$. Let $T$ be the set of all finite rational places of $K$. Since $H$ acts on the finite rational places by affine transformations, we have $\sigma(T)=T$ and $\sigma(S)=S$ for every $\sigma\in H$. Moreover,
\[
   \sum_{Q\in T}\deg Q=q.
\]
Choose a prime $\ell\mid (q-1)$, and let $K_T^S/K$ be the maximal pro-$\ell$ extension in which $P_\infty$ splits completely and which is unramified outside $T$. Write $G_T^S=\Gal(K_T^S/K)$. Since $\ell\mid (q-1)$, Lemma~\ref{lem:rational-ray-rank} gives
\[
   d(G_T^S)\ge q-1\ge4.
\]
Thus $G_T^S$ is infinite by Theorem~\ref{thm:tame-GS}. Since $\deg P_\infty=1$, we have $\gcd(\deg P_\infty,\ell)=1$. Hence all  hypotheses of Proposition~\ref{prop:restricted-lifting} are satisfied, and that proposition gives the required sequence with
\[
   \frac{|G_i|}{g_{F_i}-1}
   \ge \frac{2q(q-1)}{q-2}.
\]
\end{proof}

Thus, for every fixed \(q\ge5\), there exist function fields of arbitrarily large genus admitting solvable automorphism groups whose orders grow at least linearly with the genus, with the above coefficient asymptotic to \(2q\) as \(q\to\infty\).

\begin{theorem}
\label{thm:tame-solvable-lower-small}
For each $q\in\{2,3,4\}$, there is a sequence of function fields $F_i/\F_q$ with $g_{F_i}\to\infty$ and  solvable subgroups $G_i\le\Aut(F_i/\F_q)$ such that
\[
   |G_i|\ge L_q(g_{F_i}-1),
\]
where
\[
   L_2=\frac6{11},\qquad
   L_3=\frac{24}{11}\quad \text{and }
   \quad
   L_4=\frac{20}{3}.
\]
\end{theorem}

\begin{proof}
\noindent\emph{The case $q=2$.}
Let $K=\F_2(x)$. Then $g_K=0$ and
\[
H:=\Aut(K/\F_2)
=
\left\langle
\sigma_1:x\mapsto x+1,\;
\sigma_2:x\mapsto \frac1x
\right\rangle
\cong \PGL_2(\F_2)\cong S_3.
\]
In particular, $H$ is solvable and $|H|=6$. Let $P_2$ be the unique place of $K$ of degree $2$, corresponding to the irreducible polynomial $x^2+x+1$, and let $S:=\{P_2\}$. Since $P_2$ is the unique place of degree $2$, it is fixed by $H$. 

Let $T_3$ be the set of all places of degree $3$. Thus $\sigma(T_3)=T_3$ for every $\sigma\in H$. Let \(
T_6=\{P_{6,1},P_{6,2},P_{6,3}\},
\) where $P_{6,1},P_{6,2},P_{6,3}$ correspond respectively to the following irreducible polynomials:
\[
 f_{6,1}(x):=x^6+x^3+1,
 \quad
   f_{6,2}(x):=x^6+x^4+x^3+x+1,
   \quad
   f_{6,3}(x):=x^6+x^5+x^3+x^2+1.
\]
For $T_6$, one checks that
\[
   f_{6,1}(x+1)=f_{6,2}(x),\qquad
   f_{6,3}(x+1)=f_{6,3}(x),
\]
and
\[
   x^6f_{6,1}(1/x)=f_{6,1}(x),\qquad
   x^6f_{6,2}(1/x)=f_{6,3}(x).
\]
Thus $\sigma_1(P_{6,1})=P_{6,2}$ and $\sigma_1(P_{6,3})=P_{6,3}$, while $\sigma_2(P_{6,1})=P_{6,1}$ and $\sigma_2(P_{6,2})=P_{6,3}$. Since
$\sigma_1^2=\sigma_2^2=1$, we also have $\sigma_1(P_{6,2})=P_{6,1}$ and $\sigma_2(P_{6,3})=P_{6,2}$. Thus $\sigma(T_6)=T_6$ for every $\sigma\in H$. Let $T:=T_3\cup T_6$. Hence $\sigma(T)=T$ for every $\sigma\in H$.

Take $\ell=7$. Let $K_T^S/K$ be the maximal pro-$7$ extension in which $P_2$ splits completely and which is unramified outside $T$, and write $G_T^S=\Gal(K_T^S/K)$. Since $7\mid (2^3-1)$ and $7\mid (2^6-1)$, Lemma~\ref{lem:rational-ray-rank} gives
\[
   d(G_T^S)\ge |T_3|+|T_6|=2+3=5.
\]
Hence Theorem~\ref{thm:tame-GS} implies that $G_T^S$ is infinite. Moreover,
\[
   \sum_{Q\in T}\deg Q
   =2\cdot3+3\cdot6
   =24.
\]
Since $\deg P_2=2$ and $\gcd(2,7)=1$, all  hypotheses of Proposition~\ref{prop:restricted-lifting} are satisfied. It therefore gives the required sequence with
\[
   \frac{|G_i|}{g_{F_i}-1}\ge
   \frac{2\cdot6}{-2+24}=\frac6{11}.
\]

\medskip
\noindent\emph{The case $q=3$.}
Let $K=\F_3(x)$. Then $g_K=0$ and
\[
   H:=\Aut(K/\F_3)
   \cong \PGL_2(\F_3)
   \cong S_4.
\]
In particular, $H$ is solvable and $|H|=24$. Let $S$ be the set of all places of $K$ of degree $2$, and let $T$ be the set of all places of degree $3$. Since automorphisms of $K/\F_3$ preserve the degree of places, we have $\sigma(T)=T$ and $\sigma(S)=S$ for every $\sigma\in H$. Moreover,
\[
   |S|=\frac{3^2-3}{2}=3,
   \qquad
   |T|=\frac{3^3-3}{3}=8.
\]
Take $\ell=13$, and let $K_T^S/K$ be the maximal pro-$13$ extension in which every place of $S$ splits completely and which is unramified outside $T$. Write $G_T^S=\Gal(K_T^S/K)$. Since $13\mid(3^3-1)$, Lemma~\ref{lem:rational-ray-rank} gives
\[
   d(G_T^S)\ge |T|-(|S|-1)=8-2=6.
\]
Since $6>2+2\sqrt3$, Theorem~\ref{thm:tame-GS} implies that $G_T^S$ is infinite. Furthermore,
\[
   \sum_{Q\in T}\deg Q=8\cdot3=24.
\]
Every $P\in S$ has degree $2$, and $\gcd(2,13)=1$. Thus all hypotheses of Proposition~\ref{prop:restricted-lifting} are satisfied. It gives the required sequence with
\[
   \frac{|G_i|}{g_{F_i}-1}\ge
   \frac{2\cdot24}{-2+24}=\frac{24}{11}.
\]

\medskip
\noindent\emph{The case $q=4$.}
Let $K=\F_4(x)$, where $\F_4=\F_2(\omega)$ and $\omega^2+\omega+1=0$. Consider
\[
   H:=
   \left\langle
      \sigma:x\mapsto\frac{\omega}{x},
      \quad
      \tau:x\mapsto x+1
   \right\rangle.
\]
Since $\sigma^2=\tau^2=1$ and $\sigma\tau:x\mapsto\omega/(x+1)$ has order $5$, we have $H\cong D_{10}$. In particular, $H$ is solvable and $|H|=10$.

Let $P_2$ be the degree-two place corresponding to $f_2(x)=x^2+x+\omega$, and let $S:=\{P_2\}$. If $\alpha$ is a root of $f_2$, then $\alpha^2+\alpha=\omega$, and hence
\[
   \sigma(\alpha)=\frac{\omega}{\alpha}=\alpha+1,
   \qquad
   \tau(\alpha)=\alpha+1.
\]
Since the two roots of $f_2$ are $\alpha$ and $\alpha+1$, both $\sigma$ and $\tau$ preserve $P_2$. Thus $\sigma(S)=S$ for every $\sigma\in H$.

Let $T$ be the set of all rational places of $K$. Then $T$ is $H$-stable and
\[
   \sum_{Q\in T}\deg Q=5.
\]
Take $\ell=3$, and let $K_T^S/K$ be the maximal pro-$3$ extension in which $P_2$ splits completely and which is unramified outside $T$. Write $G_T^S=\Gal(K_T^S/K)$. Since every $Q\in T$ is rational and $3\mid(4-1)$, Lemma~\ref{lem:rational-ray-rank} gives
\[
   d(G_T^S)\ge |T|-1=4.
\]
Hence $G_T^S$ is infinite by Theorem~\ref{thm:tame-GS}. Since $\deg P_2=2$ and $\gcd(2,3)=1$, Proposition~\ref{prop:restricted-lifting} applies and gives the required sequence with
\[
   \frac{|G_i|}{g_{F_i}-1}\ge
   \frac{2\cdot10}{-2+5}=\frac{20}{3}.
\]
\end{proof}

\subsection{Unramified towers over Hermitian function fields}
Let $K/\F_q$ be a function field, let $P$ be a rational place of $K$, and let $\ell\ne p$ be a prime.

\begin{lemma}\label{lem:unramified-generator-rank}
Let $S=\{P\}$ and $T=\varnothing$, and let $J_K$ be the Jacobian of $K$. Then
\[
   d\bigl(G_{\varnothing}^{\{P\}}\bigr)
   =
   d_\ell(\operatorname{Cl}^0(K))
   =
   \dim_{\F_\ell}J_K(\F_q)[\ell].
\]
If $\mu_\ell\subset\F_q$ and \(
   \dim_{\F_\ell}J_K(\F_q)[\ell]\ge5,
\) then $G_{\varnothing}^{\{P\}}$ is infinite.
\end{lemma}

\begin{proof}
By \eqref{eq:generator-ray-class-rank}, \(
   d\bigl(G_{\varnothing}^{\{P\}}\bigr)
   =
   d_\ell(\operatorname{Cl}_K(\{P\})).
\) Since $P$ is rational, the class group of the ring of $\{P\}$-integers is naturally identified with the degree-zero divisor class group:
\begin{equation}\label{eq:P-class-degree-zero}
   \operatorname{Cl}_K(\{P\})
   \cong
   \operatorname{Cl}^0(K).
\end{equation}
Together with \eqref{eq:class-jacobian}, this gives \(
   \operatorname{Cl}_K(\{P\})
   \cong
   \operatorname{Cl}^0(K)
   \cong
   J_K(\F_q).
\) Since these groups are finite,  \(
   d_\ell (J_K(\F_q))
   =
   \dim_{\F_\ell}J_K(\F_q)[\ell].
\) Therefore
\[
   d\bigl(G_{\varnothing}^{\{P\}}\bigr)
   =
   d_\ell(\operatorname{Cl}^0(K))
   =
   \dim_{\F_\ell}J_K(\F_q)[\ell].
\]
Finally, suppose that $\mu_\ell\subset\F_q$. Then, by Theorem~\ref{thm:tame-GS},
\[
   \dim_{\F_\ell}J_K(\F_q)[\ell]
   \ge 5
   >
   2+2\sqrt{|S|+\theta_{K,\varnothing}}
   =
   2+2\sqrt2,
\]
so $G_{\varnothing}^{\{P\}}$ is infinite.
\end{proof}
Let $q=r^2$, where $r\ge3$ is a prime power, and consider the Hermitian function field
\[
   K_H=\F_q(x,y),\qquad y^r+y=x^{r+1}.
\]
Its genus is
\[
   g_{K_H}=\frac{r(r-1)}2.
\]
The Hermitian function field is $\F_q$-maximal \cite[Section~10.1]{HirschfeldKorchmarosTorres}. Hence Corollary~\ref{cor:extremal-class-group} gives
\[
   J_{K_H}(\F_q)
   \cong
   (\Z/(r+1)\Z)^{2g_H}.
\]
In particular, for every prime $\ell\mid (r+1)$,
\[
   \dim_{\F_\ell}J_{{K_H}}(\F_q)[\ell]=2g_{K_H}=r(r-1).
\]
Let $P_\infty$ be the unique place at infinity in the Hermitian function field. Its stabilizer $H_r$ in $\Aut(K_H/\F_q)\cong\PGU(3,r)$ has order
\[
   |H_r|=r^3(r^2-1).
\]
This fact can be found in \cite[Satz~5]{Stichtenoth1973AutomorphismII}.  Actually, \(H_r\) is solvable.Actually, \(H_r\) is solvable, since it is a semidirect product of a \(p\)-group by a cyclic group \cite[Proposition~3.8.5]{Stichtenoth2009}. 

\begin{theorem}
\label{thm:hermitian-solvable-lower}
Let $q=r^2$ with $r\ge3$. Then there is a sequence of function fields $F_i/\F_q$ with $g_{F_i}\to\infty$ and  solvable subgroups $G_i\le\Aut(F_i/\F_q)$ such that
\[
   |G_i|\ge\frac{2r^3(r-1)}{r-2}(g_{F_i}-1)
\]
for every $i$.
\end{theorem}

\begin{proof}
Choose a prime $\ell\mid (r+1)$. Then $\ell\ne p$ and $\mu_\ell\subset\F_q$.  The above result gives
\[
   \dim_{\F_\ell}\operatorname{Cl}^0(K_H)[\ell]
   =\dim_{\F_\ell}J_{K_H}(\F_q)[\ell]
   =2g_{K_H}=r(r-1)\ge6.
\]
Lemma~\ref{lem:unramified-generator-rank} shows that the maximal $P_\infty$-split unramified pro-$\ell$ extension of $K_H$ is infinite. Proposition~\ref{prop:restricted-lifting} applied with $T=\varnothing$ and $H=H_r$ gives a sequence with
\[
   \frac{|G_i|}{g_{F_i}-1}
   \ge\frac{r^3(r^2-1)}{r(r-1)/2-1}
   =\frac{2r^3(r-1)}{r-2}.
\]
\end{proof}
\begin{remark}
For square constant fields $\F_q$ for $q=r^2$ with $r\ge3$, the unramified tower of Hermitian construction gives solvable automorphism groups whose order grows linearly with the genus, with a coefficient of order \(q^{3/2}\) as \(q\to\infty\).
\end{remark}

\subsection{Recursive towers}
\label{sec:recursive-tower-solvable}

The preceding lower bounds are obtained from class field towers. To obtain more explicit asymptotic constructions, we also consider recursive towers with large solvable automorphism groups. We show that a suitable variant of the Garcia-Stichtenoth tower yields asymptotically large solvable automorphism groups.

Let \(q=p^e\), where \(p>2\), and define
\begin{equation}\label{eq:recursive-GS-descent}
   F_0=\F_q(x_0),\qquad
   F_{n+1}=F_n(x_{n+1}),\qquad
   x_{n+1}^q-x_{n+1}
   =\frac{x_n^q}{1-x_n^{q-1}}.
\end{equation}
After extending the constant field to \(\F_{q^2}\), this is the Garcia-Stichtenoth tower considered in \cite[Section~7.4]{Stichtenoth2009}. Let
\begin{equation}\label{eq:recursive-tu}
   t_0:=x_0^q-x_0,\qquad
   u_0:=t_0^{q-1}+1,\qquad
   K:=\F_q(u_0).
\end{equation}
By \cite[Lemma~7.4.14(a)]{Stichtenoth2009},
\[
   F_0/\F_q(t_0),\qquad
   \F_q(t_0)/K,\qquad
   F_0/K
\]
are Galois extensions of degrees \(q\), \(q-1\), and \(q(q-1)\), respectively.  For \(n\ge1\), let \(E_n\) be the Galois closure of \(F_n/K\).

\begin{proposition}
\label{prop:recursive-tower-solvable-lower}
Let \(q=p^e\) with \(p>2\).  There is a sequence of function fields \(E_n/\F_q\) with \(g_{E_n}\to\infty\), together with  solvable subgroups \(
   G_n\le\Aut(E_n/\F_q),
\) such that
\[
   |G_n|\ge(q-1)(g_{E_n}-1)
\]
for every \(n\ge1\).
\end{proposition}

\begin{proof}
Let $G_n^*$ be the Galois closure of $F_n\F_{q^2}/K\F_{q^2}$.  This is the Galois closure tower considered in \cite[Theorem~7.4.15]{Stichtenoth2009}.  Write
\[
   m_n:=[G_n^*:\F_{q^2}(t_0)].
\]
By \cite[Theorem~7.4.15 and
Corollary~7.4.16(a)]{Stichtenoth2009}, $\F_{q^2}$ is the full constant field of $G_n^*$, $m_n$ is a power of $p$, and
\begin{equation}\label{eq:GS-closure-data}
   [G_n^*: K\F_{q^2}]=(q-1)m_n,
   \qquad
   g(G_n^*)-1
   <
   m_n
\end{equation}

Assume that \(E_n\) is the Galois closure of \(F_n/K\).  The Galois closure of \(
   F_n\F_{q^2}/K\F_{q^2}
\) is  \(
   \prod_{\tau}
   \tau\!\left(F_n\F_{q^2}\right),
\) where \(\tau\) runs through the \(K\F_{q^2}\)-embeddings of \(F_n\F_{q^2}\) into a fixed algebraic closure of \(K\F_{q^2}\). Since \(\F_q\) is the full constant field of \(F_n\), we have \(
   F_n\cap K\F_{q^2}=K.
\) Thus restriction induces a bijection between the \(K\F_{q^2}\)-embeddings of \(F_n\F_{q^2}\) and the \(K\)-embeddings of \(F_n\), and
\begin{equation}\label{eq:constant-extension-closure}
  G_n^*
  = \prod_{\tau}
    \tau\!\left(F_n\F_{q^2}\right)
  =
  \left(\prod_{\sigma}\sigma(F_n)\right)\F_{q^2}
  =
  E_n\F_{q^2}.
\end{equation}
where \(\sigma\) runs through the \(K\)-embeddings of \(F_n\) into a fixed algebraic closure of \(K\).

We next determine the constant field of $E_n$.  Since $E_n\subseteq G_n^*$, its full constant field is either $\F_q$ or $\F_{q^2}$.  We claim by induction that $E_i/F_0$ is a Galois extension of $p$-power degree for every $i$. Since $F_0/K$ is Galois, we have $E_0=F_0$.  Assume that $E_i/F_0$ is a Galois extension of $p$-power degree.  Now $E_i\subseteq E_{i+1}$, since $E_{i+1}/K$ is Galois and contains $F_i$.  Moreover, $E_i/K$ is Galois.  Hence, for every \(
   \tau\in\Gal(E_{i+1}/K),
\) we have \(
   \tau(E_i)=E_i
\) and therefore \(
   \tau(F_i)\subseteq E_i.
\) Since \(
   \tau(F_{i+1})/\tau(F_i)
\) is again a Galois extension of $p$-power degree, its base change to $E_i$ gives a Galois extension of $p$-power degree
\[
   E_i\tau(F_{i+1})/E_i.
\]
By the definition of the Galois closure, $E_{i+1}$ is the compositum of all $K$-conjugates of $F_{i+1}$.  Since $E_i\subseteq E_{i+1}$,we may write
\[
   E_{i+1}
   =
   \prod_{\tau\in\Gal(E_{i+1}/K)}
   E_i\tau(F_{i+1}).
\]
Thus $E_{i+1}/E_i$ is a Galois extension of $p$-power degree. Moreover, since $E_{i+1}/K$ is Galois and \(
   K\subseteq F_0\subseteq E_{i+1},
\) the extension $E_{i+1}/F_0$ is Galois.  By the induction hypothesis, \(
   [E_{i+1}:F_0]
   =
   [E_{i+1}:E_i][E_i:F_0]
\) is a power of $p$.  This completes the induction.  If $\F_{q^2}\subseteq E_n$, then \(
   F_0\F_{q^2}\subseteq E_n.
\) Since $F_0$ has full constant field $\F_q$,
\(
 [F_0\F_{q^2}:F_0]=2\text{ divides } [E_n:F_0]=p^{s_n},
\) which is impossible since $p>2$.  Thus the full constant field of $E_n$ is $\F_q$.

It follows from \eqref{eq:constant-extension-closure} that extension of constants from $\F_q$ to $\F_{q^2}$ preserves the relevant degree and genus.  Hence \eqref{eq:GS-closure-data} gives
\begin{equation}\label{eq:descended-degree}
   [E_n:K]=(q-1)m_n
\end{equation}
and
\begin{equation}\label{eq:descended-genus}
   g_{E_n}-1
   <m_n.
\end{equation}
Since $E_n/F_0$ is a Galois $p$-extension and $F_0/K$ is Galois with solvable Galois group, the exact sequence
\[
   1\longrightarrow
   \Gal(E_n/F_0)
   \longrightarrow
   \Gal(E_n/K)
   \longrightarrow
   \Gal(F_0/K)
   \longrightarrow 1
\]
shows that $\Gal(E_n/K)$ is solvable.  Since $\F_q$ is the full constant field of $E_n$, \(
   \Gal(E_n/K)\le\Aut(E_n/\F_q).
\) Moreover, by \eqref{eq:descended-degree} and \eqref{eq:descended-genus},
\[
   |\Gal(E_n/K)|
   =(q-1)m_n
   >(q-1)(g_{E_n}-1).
\]
Since $g_{F_n}\to\infty$ and $F_n\subseteq E_n$, it follows that $g_{E_n}\to\infty$. This completes the proof.
\end{proof}

\section{Further results}\label{sec:further-applications}
Throughout this section, we write $q=p^e$, where $p=\charac(\F_q)$. A finite group is called a $p'$-group if its order is prime to $p$.  

The methods of the preceding two sections yield several further  results. First, the  construction of Proposition~\ref{prop:restricted-lifting} can be used to improve the lower bounds for $p'$-subgroup automorphism groups. Second, the upper bound  for solvable groups extends naturally to  $p$-solvable groups.  

We first discuss an application of Proposition~\ref{prop:restricted-lifting}  to \(p'\)-subgroups of \(\Aut(F/\F_q)\).

\subsection{Prime-to-p subgroups}
\label{subsec:Bq-application}
 For a function field $F/\F_q$ of genus at least two, let
\[
   b_F:=\max\bigl\{|A|: \text{A is a $p'$-subgroup of $\Aut(F/\F_q)$ }\},
\]
and define \(
   B_q:=\limsup_{g_F\to\infty}\frac{b_F}{g_F},
\) where $F$ runs through function fields with full constant field $\F_q$. Ma and Xing prove that $B_q\ge2/3$ for every $q$ \cite[Section~4]{MaXing2019}. However, it is not clear whether the constant
fields in their tower are all equal to $\F_q$, see Remark~\ref{rem:Bq-characteristic-two} for details.

We next apply Proposition~ \ref{prop:restricted-lifting} to improve the lower bounds
for $B_q$.  We begin with the following direct consequence of Proposition~\ref{prop:restricted-lifting}.

\begin{corollary}
\label{cor:pprime-lifting}
Assume the hypotheses of Proposition~\ref{prop:restricted-lifting}.  Suppose in addition that
\[
   \ell\ne p
   \qquad\text{and}\qquad
  \text{$H$ is a $p'$-group}.
\]
Then there is a sequence of function fields $E_n/\F_q$ with $g_{E_n}\to\infty$, together with $p'$-subgroups \(
   A_n\le \Aut(E_n/\F_q)
\) , such that
\[
   \frac{|A_n|}{g_{E_n}-1}
   \ge
   \frac{2|H|}
   {\,2g_K-2+\displaystyle\sum_{Q\in T}\deg Q\,}
\]
for every $n$.
\end{corollary}

\begin{proof}
By Proposition~\ref{prop:restricted-lifting}, there are finite Galois extensions $E_n/K^H$ and exact sequences
\[
   1\longrightarrow \Gal(E_n/K)
   \longrightarrow \Gal(E_n/K^H)
   \longrightarrow H
   \longrightarrow 1,
\]
where $\Gal(E_n/K)$ is an $\ell$-group.  Set \(
   A_n:=\Gal(E_n/K^H).
\) Since $\ell\ne p$ and $p\nmid |H|$, we have $p\nmid |A_n|$.  Moreover, $A_n$ is solvable, since its normal subgroup $\Gal(E_n/K)$ is an $\ell$-group and its quotient is $H$.  The remaining assertions, including $g_{E_n}\to\infty$ and
\[
   \frac{|A_n|}{g_{E_n}-1}
   \ge
   \frac{2|H|}
   {\,2g_K-2+\displaystyle\sum_{Q\in T}\deg Q\,},
\]
follow directly from Proposition~\ref{prop:restricted-lifting}.
\end{proof}
We use tamely ramified towers over the rational function field that were established in the preceding section,  and replace the solvable group by a suitable $p'$ subgroup.

\begin{theorem}\label{thm:Bq-lower}
For every $q\ge5$,
\[
   B_q\ge\frac{2(q-1)}{q-2}.
\]
Moreover,  \(  B_2\ge\frac{3}{11},
   B_3\ge\frac{8}{11}\text{ and }
   B_4\ge\frac{10}{3}.
\)

\end{theorem}

\begin{proof}
Let $q\ge5$.  Retain the function field $K=\F_q(x)$, the sets $S,T$, and the infinite pro-$\ell$ tower from the proof of Theorem~\ref{thm:tame-solvable-lower-general}.  Thus \(
   S=\{P_\infty\}\text{ and }
   \sum_{Q\in T}\deg Q=q.
\) Replace the affine group used there by
\[
   H_0=\{x\mapsto ax:a\in\F_q^\times\}.
\]
Then $|H_0|=q-1$, the group $H_0$ preserves both $S$ and $T$, and $p\nmid |H_0|$.  Hence Corollary~\ref{cor:pprime-lifting} gives
\[
   \frac{|A_n|}{g_{E_n}-1}
   \ge
   \frac{2(q-1)}{-2+q}
   =
   \frac{2(q-1)}{q-2}.
\]
For $q=2$, retain the construction from the $q=2$ case in the proof of Theorem~\ref{thm:tame-solvable-lower-small}.  Thus $\ell=7$, \(
   \sum_{Q\in T}\deg Q=24,
\) and the corresponding pro-$7$ group $G_T^S$ is infinite.  The sets $S$ and $T$ are stable under \(
   \PGL_2(\F_2)\cong S_3,
\) and hence under its cyclic subgroup $H_0\cong C_3$.  Since $p\nmid |H_0|$,  Corollary~\ref{cor:pprime-lifting} gives
\[
   \frac{|A_n|}{g_{E_n}-1}
   \ge
   \frac{2\cdot3}{-2+24}
   =
   \frac3{11}.
\]
For $q=3$, retain the construction from the $q=3$ case in the proof of Theorem~ \ref{thm:tame-solvable-lower-small}.  Thus $\ell=13$, \(
   \sum_{Q\in T}\deg Q=24,
\) and $G_T^S$ is infinite.  Since the sets $S$ and $T$ are stable under \(
   \PGL_2(\F_3)\cong S_4,
\) they are also stable under a Sylow $2$-subgroup $H_0\le S_4$ of order $8$.  Therefore Corollary~\ref{cor:pprime-lifting} yields
\[
   \frac{|A_n|}{g_{E_n}-1}
   \ge
   \frac{2\cdot8}{-2+24}
   =
   \frac8{11}.
\]
For $q=4$, retain the pro-$3$ tower and the group $D_{10}$ from the $q=4$ case in the proof of
Theorem~ \ref{thm:tame-solvable-lower-small}.  Its cyclic subgroup $H_0\cong C_5$ preserves the prescribed splitting and ramification sets. Since \(
   \sum_{Q\in T}\deg Q=5,
\) Corollary~\ref{cor:pprime-lifting} gives
\[
   \frac{|A_n|}{g_{E_n}-1}
   \ge
   \frac{2\cdot5}{-2+5}
   =
   \frac{10}{3}.
\]
In each case $b_{E_n}\ge |A_n|$.  Since $g_{E_n}\to\infty$, replacing $g_{E_n}-1$ by $g_{E_n}$ does not change the limiting lower bound. The assertions follow.
\end{proof}

\begin{remark}
\label{rem:Bq-characteristic-two}
For the cubic tower in \cite[Theorem~4.5]{MaXing2019}, when $q=2^{2m+1}$, the Galois closure may involve the quadratic constant extension $\F_{q^2}$.  Thus the corresponding Galois group need not have order prime to $p=2$, and some additional care is needed in this case. On the other hand, for every $q\ge3$, the construction above already gives a lower bound strictly larger than $2/3$.
\end{remark}

\subsection{\texorpdfstring{$p$}{p}-solvable automorphism groups}
\label{subsec:p-solvable}
We first recall some standard facts about finite groups.

Every chief factor \(A\) of a finite group \(G\) is of the form \(
   A\cong S^m
\) for some finite simple group \(S\) and some integer \(m\ge1\) \cite[Theorem~3.3.15]{Robinson1996}.

A finite group is $p$-solvable, if each of its chief factors is either an elementary abelian $p$-group or a $p'$-group \cite[Section~9.3]{Robinson1996}. The latter group may be either an elementary abelian \(\ell\)-group \(C_\ell^{\,m}\) for some prime \(\ell\neq p\), or a direct product \(S^m\) of isomorphic non-abelian finite simple groups, where \(p\nmid |S|\). In particular, every solvable group is $p$-solvable.

We begin with the rational function field and determine the maximal possible order of a \(p\)-solvable subgroup of its \(\F_q\)-automorphism group \(
\Aut\bigl(\F_q(x)/\F_q\bigr)\cong \PGL_2(\F_q).
\)
\begin{proposition}
\label{prop:PGL2-psolvable}
Let $K=\F_q(x)$, where $q=p^e$, and let $H\le\Aut(K/\F_q)$ be a  $p$-solvable subgroup.  Then
\[
   |H|\le R_q,
\]
where \(R_q\) is given by \Cref{eq:Rq-definition}.
\end{proposition}

\begin{proof}
We first recall the classical fact that every subgroup of a \(p\)-solvable group is again \(p\)-solvable. Indeed, if \(H\le G\), then intersecting a chief series of \(G\) with \(H\) and refining gives a chief series of \(H\), all of whose chief factors are either \(p\)-groups or \(p'\)-groups.

We use the subgroup classification of $\PGL_2(\F_q)$ in \cite[Appendix~A, Theorem~A.8]{HirschfeldKorchmarosTorres}.  The solvable subgroups are already bounded by $R_q,$ see   \Cref{prop:PGL2-solvable}.  Thus it remains to consider the non-solvable groups in the list.
$$\PSL_2(\F_{p^f}), \PGL_2(\F_{p^f}) \text{ with } p^f\ge4  \quad \text{and} \quad A_5 \text{ with } 5\mid (q^2-1)$$
The group $\PSL_2(\F_{p^f})$ is a non-abelian simple group \cite[Theorem~8.13]{Rotman1995} or  \cite[Appendix~A, Example~A.7(5)]{HirschfeldKorchmarosTorres}, and its order is divisible by $p$ but is not a power of $p$. Hence it is neither a $p$-group nor a $p'$-group, therefore $\PSL_2(\F_{p^f})$  with $p^f\ge4$ is not $p$-solvable.  Thus every subgroup containing $\PSL_2(\F_{p^f})$ with $p^f\ge4$ is not $p$-solvable.  Consequently, neither $\PSL_2(\F_{p^f})$ nor $\PGL_2(\F_{p^f})$ is $p$-solvable, since \(
   \PSL_2(\F_{p^f})\triangleleft \PGL_2(\F_{p^f}).
\) 

The remaining exceptional possibility is $A_5$ for $5\mid (q^2-1)$ .  Since $A_5$ is a non-abelian simple group, if it is $p$-solvable, then $A_5$ is a $p'$-group, thus $p\nmid60$, and hence $p\ge7$.  For $q=7,8$ one has $5\nmid (q^2-1)$, so $A_5$ does not occur.  Thus  $q\ge11$, and therefore \(
   |A_5|=60\le q(q-1)=R_q.
\)
\end{proof}

We now apply the method developed for solvable groups to the \(p\)-solvable case. The only additional issue is that a chief factor may be a direct product of isomorphic non-abelian simple \(p'\)-groups.

\begin{theorem}\label{thm:p-solvable}
Let $F/\F_q$ be a function field of genus $g_F\ge2$ with full constant field $\F_q$, where $q=p^e$, and let \(
   G\le\Aut(F/\F_q)
\) be a finite $p$-solvable subgroup.   Let $D_q$ be as defined in Corollary~\ref{cor:rational-exact}, and set
\[
C_q^{(p\text{-sol})}:=R_q\max\{84,D_q\}.
\]
Then
\[
|G|\le C_q^{(p\text{-sol})}(g_F-1).
\]
In particular, for every $q\ge4$,
\[
   |G|\le 21q^2(q-1)(g_F-1).
\]
\end{theorem}
\begin{proof}
Put \(F_0:=F^G\).  By \Cref{prop:g0-ge2,prop:g0-eq1}, the cases \(g_{F_0}\ge1\) give
\[
   |G|\le4(g_F-1),
\]
so it remains to consider \(g_{F_0}=0\). Use the notation of \Cref{thm:precise-rational}.  Thus choose a chief series
\[
   G=G_0\supsetneq G_1\supsetneq\cdots\supsetneq G_s=1,
   \qquad F_j:=F^{G_j},
\]
and let \(i\) be maximal such that \(g_{F_i}=0\).  Set \(
   A:=G_i/G_{i+1}
     \cong\Gal(F_{i+1}/F_i).
\) Since \(G/G_i\) is \(p\)-solvable and both \(F_i\) and \(F_0\) are rational, \Cref{prop:PGL2-psolvable} gives
\begin{equation}\label{eq:psol-rational-tail}
   [F_i:F_0]\le R_q.
\end{equation}
We consider the following cases for \(A\).

First, if \(A\) is solvable, then \(A\) is elementary abelian. Starting from \eqref{eq:psol-rational-tail}, the proof of \Cref{thm:precise-rational}  yields the result in \Cref{cor:rational-exact}.
\begin{equation}\label{eq:psol-elementary-case}
 |G|\le R_qD_q(g_F-1)
\end{equation}
It remains to consider the case where \(A\) is non-solvable. Since \(G\) is \(p\)-solvable, the non-solvable chief factor \(A\) has the form \(
A\cong S^m,
\) where \(S\) is a non-abelian finite simple group with \(p\nmid |S|\). Thus \(F_{i+1}/F_i\) is tame.  Moreover \(g_{F_{i+1}}\ge2\).  Indeed, if \(g_{F_{i+1}}=1\), then the elliptic function field \(F_{i+1}\) has an \(\F_q\)-rational place $O$ . By \cite[Theorem~3.1 and Remark~2]{MaXing2023},
\[
   \Aut(F_{i+1}/\F_q)=T_{F_{i+1}}\rtimes\Aut(F_{i+1},O),
\]
where \(T_{F_{i+1}}\) is abelian and \(|\Aut(F_{i+1},O)|\mid 24\). Hence \(\Aut(F_{i+1},O)\) is solvable, and therefore so is \(\Aut(F_{i+1}/\F_q)\). This contradicts the choice of \(A\).

The tame Hurwitz bound therefore yields
\[
   |A|\le84\bigl(g_{F_{i+1}}-1\bigr).
\]
On the other hand, Riemann--Hurwitz for \(F/F_{i+1}\) gives
\[
   [F:F_{i+1}]
   \le
   \frac{g_F-1}{g_{F_{i+1}}-1}.
\]
Together with \eqref{eq:psol-rational-tail}, we obtain
\[
\begin{aligned}
   |G|
   =[F:F_{i+1}]\cdot|A|\cdot[F_i:F_0] 
   \le
   \frac{g_F-1}{g_{F_{i+1}}-1}
   \cdot84\bigl(g_{F_{i+1}}-1\bigr)\cdot R_q
   =84R_q(g_F-1).
\end{aligned}
\]
Combining this with \eqref{eq:psol-elementary-case} proves the above bound.  For \(q\ge4\), the estimate follows from \Cref{prop:PGL2-psolvable} and \Cref{eq:Rq-definition} .
\end{proof}

\section{Conclusion}\label{sec:further-questions}

Table~\ref{tab:chief-factors} summarizes the additional chief factors and the corresponding upper bounds from \(p\)-groups to solvable, \(p\)-solvable, and arbitrary finite groups.

The second column lists the additional chief factors,  Indeed, the chief factors of a  \(p\)-group are all isomorphic to \(C_p\). Compared with  \(p\)-groups, solvable groups additionally allow elementary abelian chief factors \(C_\ell^{\,m}\) for primes \(\ell\neq p\). Compared with solvable groups, \(p\)-solvable groups additionally allow non-abelian chief factors \(S^m\) with \(p\nmid |S|\). Finally, arbitrary finite groups additionally allow non-abelian chief factors \(S^m\) with \(p\mid |S|\), where \(S\) is a finite simple group. 

The third column lists the corresponding linear bounds  from  \(p\)-groups through solvable and \(p\)-solvable groups. These results naturally lead us to the following analogue of the classical Hurwitz bound for function fields over $\F_q$.

\begin{conjecture}
\label{conj:all-groups-linear}
For every prime power $q$, there exists a constant $C_q>0$ such that, for every function field $F/\F_q$ of genus $g_F\ge2$ with full constant field $\F_q$ , one has
\[
   |\Aut(F/\F_q)|\le C_q(g_F-1).
\]
\end{conjecture}

\begin{table}[htbp]\label{table:Chief factors and Upper bounds.}
  \centering
  \renewcommand{\arraystretch}{0.85}
    \caption{Chief factors and Upper bounds.}
  \begin{tabular}{
    @{}
    >{\centering\arraybackslash}p{0.19\textwidth}
    >{\centering\arraybackslash}p{0.54\textwidth}
    >{\centering\arraybackslash}p{0.19\textwidth}
    @{}
  }
    \toprule
    Group class
      & \textbf{Additional chief factors}
      & Upper bound\\
    \midrule
     \(p\)-groups
      & \(C_p\)
      & $\Cref{thm:Pq-upper}$\\[1mm]
    solvable& \(C_\ell^{\,m}\), with \(\ell\ne p\) or \(m>1\)
      & $\Cref{cor:Sq-upper}$\\[1mm]
    \(p\)-solvable
      & \(S^m\), with \(p\nmid |S|\)
      & $\Cref{thm:p-solvable}$\\[1mm]
    arbitrary finite& \(S^m\), with \(p\mid |S|\)
      & $\Cref{conj:all-groups-linear}$\\
      \bottomrule
  \end{tabular}

  \label{tab:chief-factors}
\end{table}

\section{Acknowledgement}

The second and third authors were supported in part by the National Key Research and Development Program of China under Grants 2022YFA1004900, 2023YFE0123900, and the National Natural Science Foundation of China under Grants 12361141818, 12426302, 12271084 and 12571577. The third author thanks Shaowu Zheng for an introduction to the theory of chief factors of finite groups.

\bibliographystyle{plain}
\bibliography{solvable}
\end{document}